\documentclass[12pt, openany]{amsart}
\usepackage[utf8]{inputenc}
\usepackage[T1]{fontenc}
\usepackage[pdftex]{graphicx}
\usepackage{amsmath,amssymb,amsthm}
\usepackage{mathrsfs}
\usepackage{hyperref}
\usepackage[explicit]{titlesec}
\usepackage{pgf,tikz,pgfplots}
\usetikzlibrary{arrows}
\usepackage{graphicx}
\usepackage{array}
\usepackage{color}

\titleformat{\section}{\normalfont\scshape\centering}{\thesection}{1em}{#1}
  
\newtheorem{theorem}{Theorem}
\theoremstyle{definition}
\newtheorem{definition}[theorem]{Definition}
\theoremstyle{remark}
\newtheorem{example}[theorem]{Example}
\theoremstyle{remark}
\newtheorem{remark}[theorem]{Remark}
\theoremstyle{theorem}
\newtheorem{lemma}[theorem]{Lemma}
\theoremstyle{theorem}

\theoremstyle{theorem}
\newtheorem{proposition}[theorem]{Proposition}
\theoremstyle{theorem}

\newcommand{\scalar}[2]{\left\langle\,#1\left|\,#2\right.\right\rangle}

\newcommand{\class}{\mathcal{C}}
\newcommand{\id}{\text{Id}\,}

\newcommand{\ZZ}{\mathbb{Z}}

\newcommand{\RR}{\mathbb{R}}

\renewcommand{\SS}{\mathbb{S}}

\renewcommand{\mod}{\,\text{mod}\,}

\title[Only quadrics have pseudo-caustics]{\textit{Only quadrics have pseudo-caustics}\\ --\\ on caustics of Riemannian, pseudo-Euclidean and projective billiards in higher dimensions}

\author{Corentin Fierobe}
\address{Institute of Science and Technology Austria, Am Campus 1, 3400 Klosterneuburg, Austria}
\email{corentin.fierobekoz@gmail.com}

\begin{document}
\maketitle

\begin{abstract}
This papers studies properties of caustics of different billiard models in dimension $d$ at least $3$, namely of billiards in $\RR^d$ whose law of reflection is defined by 1) a Riemannian metric projectively equivalent to the Euclidean one; 2) a constant non-degenerate quadratic form (\textit{pseudo-Euclidean billiards}); 3) a smooth field of transverse lines to the boundary defining a law of reflection (\textit{projective billiards}). Case 1) and 2) are particular cases of 3). The paper gives a necessary and sufficient condition so that if such billiards have a caustic then the latter is a quadric. In the case of pseudo-Euclidean billiards, we even show that the only billiards having a caustic are the quadrics, for which the caustics are pseudo-confocal quadrics.
\end{abstract}

\setcounter{tocdepth}{1}
\tableofcontents

\renewcommand{\mod}{\,\text{mod}\,}

\section{Introduction}

The study of the light's behaviour in a homogeneous environment bounded with mirrors is of great interest and has fundamental applications, for example in optical engineering or in fundamental physics, \textit{e.g.} to control laser beams of particles, or to study the existence of regions with no light, but also the opposite phenomenon, the so-called \textit{caustics}, which are spots along curves with an intense concentration of light. 

These systems can be modelized by mathematical billiards, as domains in which one can imagine a ray of light evolving by either going straight forward when no obstacles are on its way, or bouncing on such an obstacle according to a certain law of reflection, for example the classical law of optics \textit{angle of incidence = angle of reflection}.

This law of reflection takes various expressions, depending on which physical phenomenon one decides to study. One can for example consider that the mirrors are not perfect anymore, or that the environment is inhomogeneous which would give a \textit{twisted} version of the classical law of reflection. One can also study, instead of the trajectory of a ray of light, the behaviour of an electron submitted to a magnetic field.

The object of interest of this paper are billiard models with a generalized law of reflection, the so-called \textit{projective billiards}. They unify various laws, including the classical one in a Euclidean or Riemannian metric whose geodesics are lines. The former models were introduced and studied by Sergey Tabachnikov \cite{taba_projectif_ball, taba_projectif}. They roughly consist of a domain endowed with a smooth field of lines tranverse to the boundary: the latter defines at each point on the boundary a law of reflection passing through this point. We investigate particular fields of transverse line $L$, for which a domain is what is called \textit{$L$-symmetric} (defined below, see Definition \ref{definition:L_symmetric}).

The paper studies these projective billiard models in dimension $d\geq 3$ and focuses on the existence and properties of caustics, \textit{i.e. hypersurfaces to which any light trajectory remains tangent after successive reflections}. It questions for which billiards a caustic has to be a quadric, and relates the answer to the class of $L$-symmetric billiards which we mentionned above. More precisely we prove the following

\begin{theorem}
\label{theorem:main1}
Let $\Omega$ be a strictly convex domain in $\RR^d$ with $d\geq 3$, whose boundary is endowed with a smooth field of transve lines $L$. Assume that the corresponding projective billiard has a caustic $\Gamma$. Then $\Gamma$ is a quadric if and only if $\partial\Omega$ is $L$-symmetric.
\end{theorem}

Interestingly, projective billiards contain the class of ususal Euclidean billiards, but more generally billiards in $\RR^d$ endowed with a metric whose geodesics are lines, and also pseudo-Euclidean or Finsler billiards (to see this, one can associate to a domain its field of orthogonal lines for the corresponding notion of "\textit{metric}"). Pseudo-Euclidean billiards are billiard for which the law of reflection is defined by a constant non-degenerate quadratic form. For Finsler billiards, the law is defined by a Finsler metric which does not depend on the fiber. Hence Theorem \ref{theorem:main1} applies to all of these billiards:

\begin{theorem}
\label{theorem:main2}
Theorem \ref{theorem:main1} applies to ususal Euclidean billiards, to billiards in $\RR^d$ endowed with a metric whose geodesics are lines, to pseudo-Euclidean billiards and to Finsler billiards.
\end{theorem}

For pseudo-Euclidean billiards, Theorem \ref{theorem:main1} can be made more precise, since 1) it can be shown that any such billiard is $L$-symmetric and 2) it is known that for two quadrics $Q_1, Q_2$ of some pencil of quadrics called pseudo-confocal, one is a caustic of the other one for the pseudo-Euclidean law of reflection. We will explain in more details these arguments (see Section \ref{section:theorem_only}), which allow to prove the 

\begin{theorem}
\label{theorem:main3}
Let $\Omega$ be a pseudo-Euclidean billiard in $\RR^d$ with $d\geq 3$ whose field of normal line to the boundary is transverse to it. If $\Omega$ has a caustic $\Gamma$, then both $\Gamma$ and $\Omega$ are pseudo-confocal quadrics.
\end{theorem}

Theorems \ref{theorem:main1}, \ref{theorem:main2} and \ref{theorem:main3} are described more precisely in Section \ref{section:statements}.

The study of caustics, and more generally of invariant curves, in billiards is an intense subsject of research. A famous example is the billiard in the ellipse (resp. ellipsoid) for which it is known \cite{taba_book} that smaller confocal ellipses (resp. ellipsoids) are caustics. Applying a KAM scheme, Lazutkin \cite{Lazutkin} showed that if the boundary of a planar Euclidean billiard is striclty convex and sufficiently smooth, then there exists a large set of caustics accumulating to the boundary. Mather \cite{mather} proved that in a planar Euclidean billiard if the curvature of the boundary vanishes, then the latter has no caustics. Gutkin and Katok \cite{GutkinKatok} gave effective conditions in such billiards for which there exists caustics-free regions. These results were extended to the $2$-sphere and hyperbolic plane by Florentin, Ostrover and Rosen \cite{FlorentinOstroverRosen}. Let us mention the famous Birkhoff-Poritzky's conjecture:

\textbf{Conjecture (Birkhoff-Poritzky)}. \textit{If a convex domain contains an open set of invariant curves then it is an ellipse.}

The conjecture is still open, although some partial positive answers have been given, see for example \cite{AvilaSimoiKaloshin, Bialy, Bialy2013, BialyMironov, BialyMironovAnnals, glut_integrability, glut_integrability_bis, glut_ratint, glut_multib, KaloshinSorrentino, Koval}. 

For a billiard in dimension at least $3$, having a caustic is a more rigid property. Berger \cite{berger_caustics} proved that if such a Euclidean billiard has a caustic, then both are confocal quadrics. This result was later extended by Glutsyuk to billiards in spaces of constant curvature \cite{glut_spaceform}. Billiards in pseudo-Euclidean spaces and Minkowski metrics have been intensively studied by Dragović and Radnović \cite{AdaDragRad_minkowski, DragRad_bicent, DragRad, DragRad_minkowski}. More generally, projective billiards have been studied in any dimensions \cite{fierobe_these, fierobe_caustics, fierobe_triangular, fierobe1, glut_ratint, glut_multib, taba_projectif, taba_projectif_ball}. The results which can be found here extend previously known results \cite{fierobe_these} found by the author.

\textbf{Plan of the article.} More precise definitions of Riemannian billiards, pseudo-Euclidean billiards and projective billiards are given in \textbf{Section \ref{section:statements}}, as well as the statements of the results, namely Theorems \ref{theorem:main_riemannian}, \ref{theorem:main_pseudo_Euclidean} and \ref{theorem:main_projective}. In \textbf{Section \ref{section:admissible}}, the reader can find a precise study of the so-called \textit{admissible} hyperplanes, allowing to view \textit{à la Berger} any cone of lines tangent to a caustic as an integral subsmanifold of a certain distribution. In \textbf{Section \ref{section:theorem_symmetry}} we prove Theorem \ref{theorem:main_projective}, from which we deduce Theorem \ref{theorem:main_riemannian} and a part of Theorem \ref{theorem:main_pseudo_Euclidean}. We complete the proof of Theorem \ref{theorem:main_pseudo_Euclidean} in \textbf{Section \ref{section:theorem_only}}.

\section{Riemannian, pseudo-Euclidean and projective billiards}
\label{section:statements}

A \textit{classical billiard} on a manifold $M$ of dimension $d\geq 2$ given with a complete Riemannian metric $g$ is a domain $\Omega\subset M$ with piecewise smooth boundary. The reflection law at a smooth point $p$ of the boundary is defined as follows, where $\gamma$ is an oriented geodesic inside $\Omega$ hitting the boundary at $p$ with velocity $v\in T_pM$:

\begin{definition}[Riemannian billiard law of reflection]
\label{definition:riemannian_reflection}
We define the \textit{reflected geodesic} from $\gamma$ at $p$ as the oriented geodesic $\tilde \gamma$ passing through $p$ with velocity $\tilde v\in T_pM$ such that $\tilde v$ is obtained from $v$ by the linear orthogonal symmetry of $T_pM$ with respect to $T_p\Omega$ - that is the only linear non-trivial involution which preserves the hyperplane $T_p\partial\Omega$ and its $g$-orthogonal line.
\end{definition}

This defines a \textit{billiard flow} $\varphi^t$ for $t\in\RR$: given any point $p\in \Omega$ and a non-zero vector $v\in T_pM$, the billiard flow $\varphi^t(p,v)$ is obtain by following the geodesic with initial conditions $(p,v)$ until it eventually reflects on the boundary, and continue its trajectory, etc. It might eventually be not defined for all $t$.

\begin{example}[$d$-dimensional sphere]
\label{example:sphere}
Consider a domain $\Omega$ on the $d$-dimensional sphere $\SS^d\subset\RR^{d+1}$ equipped with the classical spherical metric $g$: the billiard flow inside $\Omega$ follow a trajectory along broken pieces of great circles - \textit{ie} intersection of hyperplanes of $\RR^{d+1}$ with $\SS^d$. Note that the canocical projection $\pi: \RR^{d+1}\setminus\{0\}\to\RR P^d$ restricts as diffeomorphism between a halph-sphere $\SS^d_+$ and a standard affine chart $\RR^d\subset\RR P^{d}$ which sends geodesics of $\SS^d$ to lines of $\RR^d$. However the pushforward $\pi_{\ast}g$ of $g$ to $\RR^d$ is not the Euclidean metric, and for any hyperplane $H\subset\RR^d$, its $\pi_{\ast}g$-orthogonal line differs from its Euclidean orthogonal line.
\end{example}

\begin{example}[Metrics on $\RR^d$ projectively equivalent to the Euclidean one]
\label{example:projectively_equivalent}
A metric $g$ on $\RR^d$ is said to be projectively equivalent to the Euclidean metric if the geodesics of $g$ are supported by straight lines. For example, the metric $\pi_{\ast}g$ obtained in Example \ref{example:sphere} is projectively equivalent to the Euclidean metric. We can therefore define a billiard flow inside an domain $\Omega\subset\RR^d$ for the metric $g$ which follows broken lines. Note that in this case also, the $g$-orthogonal line of an hyperplane $H\subset\RR^d$ may differ from its Euclidean orthogonal.
\end{example}

In this paper we study the property of so-called caustics of a billiard. They can be defined as follows:

\begin{definition}
Let $\Omega$ be a billiard in a Riemannian manifold $(M,g)$. A \textit{caustic} is a smooth hypersurface $\Gamma\subset\Omega$ such that if the billiard flow in $\Omega$ is once tangent to $\Gamma$, then it remains tangent to it after successive reflections (as long as it is defined). 
\end{definition}

We will define two other types of billiards, namely the \textit{pseudo-Euclidean} and \textit{projective billiards} for which this notion is relevant, and will state each of our result for this different models.

In the following we consider the space $\RR^d$, $d\geq 2$ endowed with its Euclidean metric $\scalar{\cdot}{\cdot}$. Let $S\subset\RR^d$ be a $\class^2$-smooth hypersurface. We fix $n:\partial\Omega\to\RR^d$ be a field of Euclidean unit normal vectors to $S$.

\textbf{Riemannian billiards}

To state our first result, we need another definition. Let $\nu:S\to\RR^d$ be the field of $g$-orthogonal vectors to $S$ such that $\scalar{\nu}{n}\equiv 1$.

\begin{definition}
We say that $S$ is \textit{$g$-symmetric} at $p\in S$ if for any $u,v\in T_pS$, we have
\begin{equation}
\label{equation:symmetry_metric}
\scalar{dn_p(u)}{d\nu_p(v)}=\scalar{dn_p(v)}{d\nu_p(u)}.
\end{equation}
We simply say that $S$ is \textit{$g$-symmetric} if the latter holds for all $p\in S$.
\end{definition}

\begin{theorem}
\label{theorem:main_riemannian}
Let $d\geq 3$ and $g$ be a Riemannian metric on $\RR^d$ whose geodesics are supported by lines. Let $\Omega\subset\RR^d$ be a $\class^2$-smooth domain such that $\Omega$ has a caustic $\Gamma$. Then the following statements are equivalent:\\
\textbf{(i)} $\Gamma$ is a quadric;\\
\textbf{(ii)} $\partial\Omega\subset\RR^d$ is $g$-symmetric.
\end{theorem}

In fact Theorem \ref{theorem:main_riemannian} has a analogous versions for so-called pseudo-Euclidean and projective billiards which we are going to define.

\textbf{Pseudo-Euclidean billiards}

Let $Q$ be a non-degenerate bilinear form on $\RR^d$. Let $\Omega\subset\RR^d$ be a domain with piecewise smooth boundary. We define the pseudo-Euclidean reflection of oriented lines with respect to $Q$ at point $p\in\partial\Omega$ as in Definition \ref{definition:riemannian_reflection}, where\\
\textbf{(i)} the geodesics are replaced by oriented lines;\\
\textbf{(ii)} the notion of $g$-orthogonality is replaced by the notion of $Q$ orthogonality: since $Q$ is non-degenerate, the $Q$-orthogonal line $L_p$ to $T_p\partial\Omega$ in $T_p\RR^d$ is well-defined. Hence if $L_p\nsubseteq T_p\Omega$, there is a unique non-trivial linear involution of $T_p\RR^d$ preserving $T_p\Omega$ and $L_p$, and the latter defines the pseudo-Euclidean reflection at $p$ with respect to $T_p\Omega$ as in Definition \ref{definition:riemannian_reflection}. When $L_p\subset T_p\Omega$, the reflection is not defined.

%
%

\begin{theorem}
\label{theorem:main_pseudo_Euclidean}
Let $d\geq 3$ and $Q$ be a non-degenerate quadratic form on $\RR^d$. Let $\Omega\subset\RR^d$ be a domain whose boundary is a $\class^2$-smooth space-time hypersurface. Then $\Omega$ has a caustic $\Gamma$ if and only if $\Omega$ and $\Gamma$ are quadrics, which are pseudo-confocal. 
\end{theorem}

\begin{remark}
The notion of pseudo-confocal quadrics extends the usual notion of confocal quadrics as a specific pencil of quadrics depending on the quadratic form associated to the space (see \cite{fierobe_these} for more details). For example if the space is endowed with the quadratic form 
$$dx_1^2+\ldots+dx_r^2-dx_{r+1}^2-\ldots-dx_{d}^2$$
then the family of quadrics depending on a parameter $\lambda$ and defined by the equations
$$\mathcal Q_{\lambda}:\quad \frac{x_1^2}{a_1-\lambda}+\ldots+\frac{x_r^2}{a_r-\lambda}+\frac{x_{r+1}^2}{a_{r+1}+\lambda}+\ldots+\frac{x_{d}^2}{a_{d}+\lambda}=1$$
where $a_1,\ldots,a_d\in\RR$ is a pseudo-confocal pencil of quadrics.
\end{remark}

\textbf{Projective billiards}

The two previous billiard models lead to a more general definition of billiard, called \textit{projective billiards}, introduced and studied by Tabachnikov \cite{taba_projectif_ball, taba_projectif}. Note that given an oriented line $\ell$ passing through a point $p\in\RR^d$ can be seen as a geodesic, hence one can associates to it its unit velocity vector $v\in T_p\RR^d$.

\begin{definition}[Projective billiards and projective law of reflection]
\label{definition:projective_law}
A \textit{projective billiard} is a domain $\Omega\subset\RR^d$ with piecewise smooth boundary, equipped with a field of transverse line $L$: that is for any smooth point $p$ of $\partial\Omega$, $L_p$ defines a line passing through $p$ and transverse to $\partial\Omega$ at $p$. We denote it by the pair $(\Omega,L)$.

Let $\ell$ be an oriented line $\ell$ hitting the boundary transversally at a point $p$ with unit velocity $v\in T_p\RR^d$. The line $\ell$ is said to be \textit{reflected into the oriented line $\ell'$ by the projective law of reflection at $p$} if\\
\textbf{(i)} $\ell'$ contains $p$;\\
\textbf{(ii)} Let $w$ be the unit velocity of $L_p$. The unit velocity vectors $v'$ of $\ell'$ is obtained as the image of the unit velocity vector $v$ of $\ell$ by the only non-trivial linear involution of $T_p\RR^d$ which preserves $T_p\partial\Omega$ and $\RR w$.

This defines, as in the case of the usual billiards, a \textit{projective billiard flow} in $\Omega$.
\end{definition}

\begin{remark}
The projective law of reflection \ref{definition:projective_law} is equivalent to the following items:\\
\textbf{(i)} the lines $\ell$, $\ell'$ and $L_p$ are contained in a $2$-dimensional plane $\mathcal P$ transverse to $T_p{\Omega}$;\\
\textbf{(ii)} if $T$ is the line $T_p\cap\mathcal P$, then the quadruple of lines $(\ell,\ell';T,L_p)$ are in involution (meaning that it is a harmonic quadruple of points in the space of lines contained in $\mathcal P$ and passing through the point $p$, identified with $\RR P^1$: their cross-ratio in this order is $-1$).
\end{remark}

\begin{remark}
Let $g$ be a metric projectively equivalent to the Euclidean one, $\Omega\subset\RR^d$ a domain with piecewise smooth boundary, and $L_p$ the $g$ orthogonal line to $T_p\partial\Omega$ passing through $p$. Then by definition the billiard flow for $g$ in $\Omega$ corresponds to the projective billiard flow in $(\Omega,L)$. The same statement hold for pseudo-Euclidean billiards with space-time boundary.
\end{remark}

Let a smooth domain $\Omega\subset\RR^d$, with $d\geq 3$, equipped with a smooth field of transverse lines to the boundary $L$. We assume that the corresponding porjective billiard $(\Omega,L)$ has a caustic $\Gamma$. 

We observe \cite{berger_caustics} that we can study the problem locally  by considering an open subset $S\subset\partial\Omega$ of the boundary and two open subsets $U,V\subset\Gamma$ of the caustic such that there is at least an oriented line tangent to $U$ and reflected on $S$ into an oriented line tangent to $V$. Note that here we consider the same projective law of reflection induced by the restriction of $L$ to $S$. Denote this structure by $(S,L)$ and call it \textit{line-framed hypersurface}. This leads to the

\begin{definition}
\label{definition:piece_of_caustic}
Let $S,U,V\subset\RR^d$ be embedded hypersurfaces of $\RR^d$ and $L$ be a field of transverse lines to $S$. We say that the pair $(U,V)$ is a \textit{piece of caustic for the projective billiard} $(S,L)$ if the set of oriented lines tangent to $U$, intersecting $S$ transversally and reflected to oriented lines tangent to $V$ by the projective law of reflection on $(S,L)$ is a non-empty open subset of the set of lines tangent to $U$.
\end{definition}
Let $S\subset\RR^d$ be a hypersurface endowed with a smooth field of transverse line $L$. As in the Riemannian case, we define $\nu:S\to\RR^d$ be a field of vectors such that $\scalar{\nu}{n}\equiv 1$ and $\nu_p$ as the same direction as $L_p$.

\begin{definition}
\label{definition:L_symmetric}
We say that $S$ is \textit{$L$-symmetric} at $p\in S$, if for any $u,v\in T_pS$, we have
\begin{equation}
\label{equation:symmetry_metric}
\scalar{dn_p(u)}{d\nu_p(v)}=\scalar{dn_p(v)}{d\nu_p(u)}.
\end{equation}
We simply say that $S$ is \textit{$L$-symmetric} if the latter holds for all $p\in S$.
\end{definition}

\begin{theorem}
\label{theorem:main_projective}
Let $d\geq 3$, $S\subset\RR^d$ be a $\class^2$-smooth embedded hypersurface with non-degenerate second fundamental form, and endowed with a field of transverse lines $L$. Then the following statements are equivalent:\\
\textbf{(i)} $U$ and $V$ are open subsets of quadrics;\\
\textbf{(ii)} $\partial\Omega\subset\RR^d$ is $L$-symmetric.\\
If (i) or (ii) is satisfied, then $U$ and $V$ are contained in the same quadric.
\end{theorem}

\textbf{Comments on the previous results}

As explained, billiards in Riemannian metric projectively equivalent to the Euclidean one and pseudo-Euclidean billiards are particular cases of projective billiards: given $\Omega$, one of the two former billiards, the corresponding field of normal lines - either for the Riemannian metric or for the pseudo-Euclidean structure - defines a ŝmooth field of lines $L_p$ on $\partial\Omega$. Moreover it is transverse to $\partial\Omega$ in the Riemannian case, and also in the pseudo-Euclidean case when $\partial\Omega$ is space-time.

Therefore, Theorem \ref{theorem:main_riemannian} and \ref{theorem:main_pseudo_Euclidean} in these two cases will be corollaries of Theorem \ref{theorem:main_projective} in the projective case: Theorem \ref{theorem:main_riemannian} is just the translation of Theorem \ref{theorem:main_projective} in a Riemannian metric, and Theorem \ref{theorem:main_pseudo_Euclidean} will use the additional result that caustics of pseudo-Euclidean billiards in quadrics are well understood.

I strongly believe that Theorem \ref{theorem:main_projective} applies to Minkowski billiards: the latter can be seen as projective billiard wich are symmetric in the sense of Definition \ref{definition:L_symmetric}. The result is in preparation. For more details about Minkowski billiards and billiards in Finsler metrics see \cite{AAFOR,GutkinTabachnikov}.

To go beyond the results presented here, I am thinking about two interesting and relevant question. The first is to ask \textit{whether a projective billiard which is not $L$ symmetric can have a caustic}. The negative answer to this question would allow to conclude that if a projective billiard has a caustic, then the latter is a quadric. The positive answer would provide interesting examples of billiards with a caustic.

For the second question, note that, in the plane, given a convex domain $\Omega'$ nested in another one $\Omega$, one can endow $\Omega$ with a field of transverse lines such that $\partial\Omega'$ is a caustic for the induced projective billiard in $\Omega$. This is true in particular in the case when $\Omega'$ is bounded by a conic. Hence if a projective billiard has a conic as a caustic, then itself can be something else than a conic. This construction is however more delicate in dimension greater than $2$. However, following Definition \ref{definition:L_symmetric} and Theorem \ref{theorem:main_projective}, a natural question arises: \textit{if a projective billiard is $L$-symmetric and has a quadric as a caustic, is it a quadric itself?} Answering this question is probably difficult, and the study of billiards in quadrics for different billiard structures might help, as it does for pseudo-Euclidean billiards.

\section{Admissible hyperplanes} 
\label{section:admissible}

Having a caustic $\Gamma$ for a projective billiard $(\Omega,L)$ is a very strong property, which has important geometric consequences on the cones of lines tangent to $\Gamma$ and intersecting $\partial\Omega$ at a same point $q$: they are integrable surfaces of a certain distribution of hyperplanes of $T_p\partial\Omega$ (see \cite{berger_caustics} for the case of Euclidean billiards). This observation still holds for projective billiards, as we will show in Proposition \ref{proposition:admissible_hyperplane}.

We consider the following \textit{important} situation. Let a line framed-hypersurface $(S,L)$ having a piece of caustic $(U,V)$. Fix a triple of points $(p,q,r)\in U\times S\times V$ such that $p$ is the tangency point of an oriented line reflected on $(S,L)$ at the point $q$ into an oriented line tangent to $V$ at $r$. \textit{The hyperplanes $T_pU$ and $T_rV$ intersects $T_qS$ into a hyperplane $H$ of $T_qS$} (a linear subspace of codimension $1$).

Hence we can do the following trivial remark: any smooth path $(p(t),q(t),r(t))\in U\times S\times V$ passing through $(p,q,r)$ at $0$ is such that $p'(0)$ belongs to the hyperplane $T_pU$, generated by $H$ and the direction of the line $pq$, and an analogous statement for $r'(0)$. This remark is of high importance when we don't assume the existence of a piece of caustic. This leads to the following

\begin{definition}
\label{definition:admissible_hyperplane}
Let $(S,L)$ be a line-framed hypersurface, $\ell$ an oriented line intersecting $S$ transversally at a point $q\in S$ and $H$ a hyperplane of $T_qS$. We say that $H$ is $\ell$-\textit{admissible} if the orthogonal projection of $\ell$ to $T_qS$ is transverse to $H$, and for any smooth path $\ell(t)$ passing through $\ell$ at $0$ and intersecting $S$ transversally at $q(t)$, there exists a curve $(p(t),q(t))$ satisfying the following:\\
\textbf{(i)} $p(0)\in\ell$ and the points $p(t),q(t),r(t)$ are not colinear;\\
\textbf{(ii)} $\ell(t)$ is reflected into the oriented line from $q(t)$ to $r(t)$ by the projective law of reflection on $(S,L)$;\\
\textbf{(iii)} $p'(0)$ belongs to the hyperplane containing the line $\ell$ and $H$;\\
\textbf{(iv)} $q'(0)$ belongs to the hyperplane containing the line $q(0)r(0)$ and $H$. 
\end{definition}

As discussed, in the situation when $(S,L)$ admits a piece of caustic $(U,V)$, we have the 

\begin{proposition}
\label{proposition:caustic_implies_admissible}
Suppose that $(S,L)$ admits a piece of caustic $(U,V)$. Then the intersection $H$ of the tangent space $T_qS$ with $T_pU$ or with $T_rV$, such that the line $pq$ is tangent to $U$, is $pq$-admissible. 
\end{proposition}

The main result of this section is described by the next proposition. We say that a line $\ell$ intersecting $S$ at $q$ is \textit{above} a vector $\xi\in T_qS$ if it projects orthogonally to $T_qS$ in a line containing $\xi$.

\begin{proposition}
\label{proposition:admissible_hyperplane}
Suppose $S$ has a non-degenerate second fundamental form at $q$. Then for all line $\ell$ intersecting $S$ transversally at $q$ and above a vector $\xi\in T_qS$, the number of $\ell$-admissible hyperplanes $H\subset T_qS$ is at most $d-1$.
\end{proposition}

In order to prove Proposition \ref{proposition:admissible_hyperplane}, we first show that admissible hyperplanes satisfies a certain linear equation, see Lemma \ref{lemma:equation_admissible}. 

\textbf{Notations.} We endow $\RR^d$ with its usual Riemannian metric: we denote by $x\cdot y = \scalar{x}{y}$ the canonical scalar product between two vectors $x,y$ of $\RR^d$. 

There is a one-to-one correspondance (up to a sign) between hyperplanes $H\subset T_qS$ and their unit normal vectors $\pm\eta\in T_qS$, so that we will write $\eta_H$ or $H_{\eta}$ for vectors and hyperplanes corresponding to eachother.

Let $n:S\to\RR$ be a smooth field of normal unit vector to $S$, and $\nu:S\to\RR^d$ be a smooth field of vectors such that $n\cot\nu\equiv1$ and $L_p$ is directed by $\nu(p)$ for any $p\in S$.

Denote by $II_q:T_qS\times T_qS\to\RR$ the second fundamental form of $S$ at $q$.

\begin{lemma}
\label{lemma:equation_admissible}
Let $\ell$ be a line intersecting $S$ transversally at $q$ and $H$ a hyperplane of $T_qS$. Denote by $\xi$ the vector of $T_q$ above $\ell$ such that $\xi+\nu(q)$ gives the direction of $\ell$. Then if $H$ is $\ell$-admissible, then $\xi\notin H$ and for any $x\in T_qS$ we have
\begin{equation}
\label{eq:principale}
\scalar{d\nu_q(x)}{\eta_H}+\scalar{\eta_H}{\nu}II_q(\nu,x)+\scalar{\eta_H}{\xi}II_q(\xi,x)=\alpha\scalar{\eta_H}{x}
\end{equation}
where $\alpha\in\RR$ depends on the problem, and in the expression $II_q(\nu,x)$, $\nu$ stands for its orthogonal projection to $T_qS$.
\end{lemma}

\begin{remark}
The proof will imply that 
$$\alpha=-\frac{1}{2}\left(\frac{\|\nu(q)+\xi\|}{pq}+\frac{\|\nu(q)-\xi\|}{qr}\right)\in\RR$$ 
is related to the distances $pq$ and $qr$ between the points $p=p(0)$, $q=q(0)$ and $r=r(0)$ of Definition \ref{definition:admissible_hyperplane}.
\end{remark}

\begin{proof}[Proof of Lemma \ref{lemma:equation_admissible}]
The idea of the proof is similar to the one which can be found in \cite{berger_caustics}, yet a bit more general. Let $H,\ell,\xi, \eta=\eta_H$ be as in the statement and suppose that $H$ is $\ell$-admissible. The idea consists of choosing in Definition \ref{definition:admissible_hyperplane} curves $q_i(t)$ on $S$ tangent to the principal direction of $S$ at $q$. 

Let $\underline u=(u_1,\ldots,u_{d-1})$ be an orthonormal basis of $T_qS$ in which $II_q$ is diagonal (the principal directions of $S$ at $q$). Fix an integer $i$ between $1$ and $d-1$. Consider a curve $q(t)$ contained in a plane transverse to $S$ such that $q'(0)=u_i$ and $u_i(t):=q'(t)$ has unit norm. In the following, when we omit the '$(t)$', it implies that $t=0$.

If $\xi = \sum_{k=1}^{d-1}\xi_ku_k$ is the expansion of $\xi$ in $\underline u$, define $\xi(t)=\sum_{k=1}^{d-1}\xi_ku_k(t)$, where for each $k\neq i$ the vector $u_k(t)\in T_{q(t)}S$ is obtained from $u_k$ by parallel transport along $q(t)$.

This defines a curve of lines $\ell(t)$ passing through $q(t)$ of direction $e(t) := \nu(q(t))+\xi(t)$. Since $H$ is $\ell$-admissible, there is a curve $(p(t),q(t))$ satisfying items (i)-(iv).

Now write $\overline e(t) = \nu(t)-\xi(t)$ directing the line $\ell'(t)$ reflected from $\ell$ at $q(t)$ and oriented to the same side of $T_{q(t)}S$ as $e(t)$. We can find real number of same sign (say positive) $a(t),c(t)>0$ such that $p(t) = q(t)+a(t)e(t)$ and $r(t)=q(t)+c(t)\overline e(t)$. The expressions
$$n_p := \scalar{\eta}{e}n(q)-\scalar{n(q)}{e}\eta = \scalar{\eta}{e}n(q)-\eta$$
$$n_r := \scalar{\eta}{\overline e}n(q)-\scalar{n(q)}{\overline e}\eta = \scalar{\eta}{\overline e}n(q)-\eta$$
defines normal vectors to the hyperplanes of $\RR^d$ generated respectively by $H$ and $e$, and by $H$ and $\overline e$. 

By items (iii)-(iv), $p'$ and $r'$ are orthogonal respectively to $n_p$ and $n_r$, which implies that 
\begin{equation}
\label{equation:items_iii_iv}
\frac{\scalar{p'}{n_p}}{a}+\frac{\scalar{r'}{n_r}}{c}=0.
\end{equation}
The first term can be computed as $\scalar{p'}{n_p}=\scalar{q'}{n_p}+a\scalar{e'}{n_p}$ since $e$ is orthogonal to $n_p$. Hence
$$\scalar{p'}{n_p} = -\scalar{u_i}{\eta}+a\scalar{d\nu(u_i)+\xi'}{n_p}.$$
Similarly $\scalar{r'}{n_r} = -\scalar{u_i}{\eta}+c\scalar{d\nu(u_i)-\xi'}{n_r}$. Hence Equation \eqref{equation:items_iii_iv} can be rewritten as
$$-\left(\frac{1}{a}+\frac{1}{c}\right)\scalar{u_i}{\eta}+\scalar{d\nu(u_i)}{n_p+n_r}+\scalar{\xi'}{n_p-n_r}=0.$$
The sums $n_p+n_r$ and $n_p-n_r$ can be computed as
$$n_p+n_r = 2\scalar{\nu(q)}{\eta}n(q)-2\eta
\qquad\text{and}\qquad 
n_p-n_r = 2\scalar{\xi}{\eta}n(q)$$
hence Equation \eqref{equation:items_iii_iv} rewrites as
\begin{multline}
-\left(\frac{1}{a}+\frac{1}{c}\right)\scalar{u_i}{\eta}+2\scalar{d\nu(u_i)}{n(q)}\scalar{\nu(q)}{\eta}-2\scalar{d\nu(u_i)}{\eta}\\
+2\scalar{\eta}{\xi}\scalar{n(q)}{\xi'}=0.
\end{multline}
Finally derivative $\xi'$ is such that $\scalar{\xi'}{n(q)}=-k_i\ell_i=-II_q(\xi,u_i)$. Moreover, since $\scalar{n}{\nu}\equiv1$ we derive $\scalar{d\nu(u_i)}{n(q)} = -II_q(\nu(q),u_i)$. Hence Equation \eqref{equation:items_iii_iv} is equivalent to
\begin{multline}
\label{equation:final_ui}
-\left(\frac{1}{a}+\frac{1}{c}\right)\scalar{u_i}{\eta}=2\scalar{\nu(q)}{\eta}II_q(\nu(q),u_i)+2\scalar{d\nu(u_i)}{\eta}\\
-2\scalar{\eta}{\xi}II_q(\xi,u_i).
\end{multline}
Equation \eqref{equation:final_ui} gives the result by replacing $u_i$ by any $x=\sum_{k=1}^{d-1}x_k u_k$.
\end{proof}

We can now prove Proposition \ref{proposition:admissible_hyperplane}.

\begin{proof}[Proof of Proposition \ref{proposition:admissible_hyperplane}]
Fix a line $\ell$ intersecting $S$ transversally at $q$, $\xi\in T_pS$ be the vector such that $\xi+\nu$ gives the direction of $\ell$. Suppose that we are given an $\ell$-admissible hyperplane $H$ of normal vector $\eta=\eta_H$.

Denote by $d\nu^{\ast}$ and $N$ the endomorphisms of $T_qS$ satisfying for any $x,y\in T_q S$
$$\scalar{d\nu(x)}{y}=\scalar{x}{d\nu^{\ast}(y)}
\qquad\text{and}\qquad
II_q(x,y)=\scalar{N(x)}{y}.$$
Consider the endomorphisms $f,g$ of $T_pS$ defined for all $x\in T_pS$ by
$$f(x) = d\nu^{\ast}(\eta)+\scalar{x}{\nu}N(\nu)+\scalar{x}{\xi}N(\xi)$$
$$g(x) = d\nu^{\ast}(\eta)+\scalar{x}{\nu}N(\nu).$$
Note that for any $x\in T_qS$ they satisfy
\begin{equation}
\label{equation:linear_construction}
f(x)=g(x)+\scalar{x}{\xi}N(\xi).
\end{equation}
By Lemma \ref{lemma:equation_admissible}, if $H$ is $\ell$-admissible, there is a certain $\alpha\in\RR$ such that $\eta$ satisfyies $f(\eta) = \alpha \eta$. This is equivalent to say that $\eta$ is an eigenvector of $f$ associated to the eigenvalue $\alpha$. Consider any such eigenvector $x$, satisfying $f(x) = \alpha x$ by definition. From \eqref{equation:linear_construction} we get
$$g(x)-\alpha x=-\scalar{x}{\xi}N(\xi).$$
Consider first the case when $\alpha$ is not a (real) eigenvalue of $g$. This means that the operator $g-\alpha \id$ is invertible and
$$x = -\scalar{x}{\nu}(g-\alpha \id)^{-1}N(\nu)\in\RR(g-\alpha \id)^{-1}N(\nu).$$
Hence the set of eigenvectors of $f$ associated to the eigenvalue $\alpha$ has dimension at most $1$.

In the case when $\alpha$ is also an eigenvalue of $g$, we have $g(x)=\alpha x$. Hence we deduce from \eqref{equation:linear_construction} that $$\scalar{x}{\xi}N(\xi) = 0.$$
Since $II_q$ is non degenerate, this implies $\scalar{x}{\xi}=0$.

Conclusion the eigenspace of $f$ associated to $\alpha$ has either dimension $\leq 1$ or contains only vectors orthogonal to $\xi$. For $\eta$, the second possibility is not possible (otherwise $\xi\in H$). This implies the result.
\end{proof}

\section{Proof of Theorem \ref{theorem:main_projective}}
\label{section:theorem_symmetry}

To prove Theorem \ref{theorem:main_projective}, we notice that we can assume $d=3$. This idea can be found in \cite{berger_caustics} and is based on the following argument: \textit{an embedded hypersurface $\Gamma\subset\RR^d$ is a quadric if and only if for any $3$-dimensional vector space $V\subset\RR^d$, the intersection $V\cap \Gamma$ is a quadric of $V$.} 

Hence it is enough to prove Theorem \ref{theorem:main_projective} for $d=3$. Let $(S,L)$ be a line-framed hypersurface of $\RR^3$ which admits a piece of caustics $(U,V)$. Given a point $q\in S$, we are interested in the set of lines tangent to $U$ and reflected to $V$ by the projective law of reflection on $S$ at $q$. We say that such a set is a \textit{piece of quadratic cone} if it lies on a cone whose sections by planes are conics. The following proposition will be enough to conclude:

\begin{proposition}
\label{proposition:cone_lines}
Suppose that $q\in S$ and the second fundamental form of $S$ at $q$ is non-degenerate. Then the following statements are equivalent:\\
\textbf{(i)} The set of lines tangent to $U$ and reflected to $V$ by the projective law of reflection on $S$ at $q$ is a piece of quadratic cone;\\
\textbf{(ii)} $S$ is $L$-symmetric at $q$.
In this case, the set of lines tangent to $V$ and containing $q$ are contained in the same cone.
\end{proposition}

\begin{proof}[Proof of Theorem \ref{theorem:main_projective} using Proposition \ref{proposition:cone_lines}]
It is an immediate consequence of the fact that $U$ is an open set of a quadric if and only if the piece of quadratic cone tangent to $U$ are quadratic. This result can be shown by duality as it is presented in \cite{berger_caustics}.
\end{proof}

It remains to show that Proposition \ref{proposition:cone_lines} holds. The proof will be given in a succession of lemmae all along this section. We will refer to the set of lines in item (i) of Proposition \ref{proposition:cone_lines} as \textit{partial cone at $q$}.

The first lemma gives a differential equation satisfied by the curve obtained by intersecting the partial cone at $q$ with an affine plane of direction $T_qS$.

\textbf{Setting.} Let $q\in S$ be a point such that the second fundamental form of $S$ at $q$ is non-degenerate. By an affine change of variables, we can find a set of coordinates $(x,y,z)$ on $\RR^3$ such that\\
 - $q$ is the origin of $\RR^3$;\\
 - $T_qS$ is given by the equation $z=0$;\\
 - the principal directions of $S$ at $B$ are given by the vectors $u_1=(1,0,0)$ ad $u_2=(0,1,0)$ with principal curvatures $k_1, k_2\in\RR^{\ast}$.

Let $(U,V)$ be a piece of caustics of $S$ and consider the smooth family of lines $\ell(t)$ containing $q$ and tangent to $U$. Let $p_0(t) = (x(t),y(t),1)$ be the point of intersection of $\ell(t)$ with the plane of equation $z=1$.

\begin{lemma}
Write $\nu(q)=(\nu_1,\nu_2,1)$, $d\nu(u_1)+k_1\nu_1\nu=(a,b,0)$ and $d\nu(u_2)+k_2\nu_2\nu=(c,d,0)$. Then
The map $t\mapsto (X(t),Y(t)):=(x(t)-\nu_1,y(t)-\nu_2)$ satisfies the following differential equation
\begin{equation}
\label{equation:eq_diff_cones}
\left(d-a+k_2Y^2-k_1X^2\right)X'Y'+(b+k_1XY)X'^2-(c+k_2XY)Y'^2=0.
\end{equation}
\end{lemma}

\begin{proof}
The differential equation is a direct computation from Lemma \ref{lemma:equation_admissible}. Indeed, consider $M$ and $N$ two by two real matrices such that for any $u,v\in T_qS$ writen in the basis $(u_1,u_2)$ we have
$$\scalar{d\nu(u)}{v}=\scalar{Mu}{v}
\qquad\text{and}\qquad
II_q(u,v)=\scalar{Nu}{v}.$$
Here $N$ is the diagonal matrix whose entries are $k_1$ and $k_2$, and $M$ will be computed later. Lemma \ref{lemma:equation_admissible} implies that $\eta$ is an eigenvalue of the linear map 
$$f:\eta\mapsto M\eta+\scalar{\nu}{\eta}N\nu+\scalar{\xi}{\eta}N\xi.$$
Its matrix $\mathcal M(f)$ in the basis $(u_1,u_2)$ can be easily computed  using the fact that $\xi=(x(t),y(t),0)$ and the different expressions for $\nu$ and $d\nu$:
\begin{equation}
\label{equation:matrix_cones}
\mathcal M(f) = \left(
\begin{matrix}
a+k_1X^2&b+k_1XY\\
c+k_2XY&d+k_2Y^2
\end{matrix}
\right).
\end{equation}
Consider $H(t)$ the hyperplane of $T_q S$ obtained by the intersection of $T_qS$ and the tangent plane $T_{p(t)}U$, where $p(t)$ is the point of tangency of $\ell(t)$ with $U$. By Proposition \ref{proposition:caustic_implies_admissible}, $H(t)$ is admissible. A normal vector to $H(t)$ is given by
$$\eta(t) = (Y'(t),-X'(t),0).$$
Indeed, the vector $p_0'(t)=(X'(t),Y'(t),0)$ is tangent at $p_0(t)$ to the cone $\mathcal C_{U,q}$ of lines tangent to $U$, and contained in a horizontal plane. Hence it is contained in $T_pU=T_{p_0}\mathcal C_{U,q}$ and $T_qS$ hence in $H(t)$.

Now by Proposition \ref{proposition:admissible_hyperplane}, $\eta(t)$ is an eigenvalue of $f$, in particular 
$$\det (\eta,\mathcal M(f)\eta) = 0.$$
Expanding the previous equation leads to the result.
\end{proof}

Let us assume for the rest of the section that $k_1,k_2>0$, the other cases being similar. To simplify Equation \eqref{equation:eq_diff_cones}, we can do a change of variables according to the following obvious result.

\begin{lemma}
Suppose $k_2,k_1>0$. Write ${X}_1=\sqrt{k_1}X$ and ${Y}_1=\sqrt{k_2}Y.$
Then $(X,Y)$ is a solution of Equation \eqref{equation:eq_diff_cones} if and only if $({X}_1,{Y}_1)$ satisfies the differential equation
\begin{equation}
\label{equation:eq_diff_cones_renormalized}
\left(d_1-a_1+y^2-x^2\right)x'y'+(b_1+xy)x'^2-(c_1+xy)y'^2=0
\end{equation}
where $a_1=a$, $d_1=d$, 
$b_1 = b\sqrt{k_2/k_1}$, and $c_1 = c\sqrt{k_1/k_2}.$
\end{lemma}

The rest of the section is devoted to the study of the set of solutions of Equation \eqref{equation:eq_diff_cones_renormalized}. Note that we cannot apply Cauchy-Lipschitz theorem due to the form of the equation. And in fact, given a point $(x_0,y_0)\in\RR^2$, there is not a unique solution of Equation \eqref{equation:eq_diff_cones_renormalized} with initial condition $(x_0,y_0)$. We prove the

\begin{lemma}
\label{lemma:number_solutions}
There exist a strict algebraic subset $V\subset\RR^2$ such that for any $(x_0,y_0)\in\RR^2\setminus V$ Equation \eqref{equation:eq_diff_cones_renormalized} has exactly two solutions with initial conditions $(x_0,y_0)$, up to reparametrization. Moreover, if $b_1=c_1$ then $V$ is the union of two conics.
\end{lemma}

\begin{proof}
We can prove the result for Equation \eqref{equation:eq_diff_cones} instead of \eqref{equation:eq_diff_cones_renormalized}. Then the result comes from the observation that solutions $\gamma=(x,y)$ of Equation \eqref{equation:eq_diff_cones} are tangent to eigenspaces of the matrix $\mathcal N(f)=R^{^-1}\mathcal M(f)R$ defined in \eqref{equation:matrix_cones}, where $R$ is the rotation by $\pi/2$ in a given orientation. The latter depends on $(x,y)$.

Let $V$ be the set of points $(x,y)$ such that $\mathcal N(f)$ has only one eigenvalue. This set is given by the point $(x,y)$ for which the discriminant $\Delta(x,y)$ of the characteristic polynomial of $\mathcal N(f)$ vanishes. From \eqref{equation:matrix_cones}, we compute that $\Delta$ is a non-zero polynomial given by 
$$\Delta(x,y) = (d-a+k_2y^2-k_1x^2)^2+4(b+k_1xy)(c+k_2xy),$$
hence $V$ is a strict algebraic subset of $\RR^2$. The condition $b_1=c_1$ is the same as $k_2b=k_1c$, and in this case $\Delta$ can be written
$$\Delta(x,y) = (d-a+k_2y^2-k_1x^2)^2+4\frac{k_1}{k_2}(c+k_2xy)^2,$$
Hence $\Delta(x,y)=0$ is given by the union of two conics defined by equations 
$$d-a+k_2y^2-k_1x^2=0
\qquad\text{and}\qquad
c+k_2xy=0.$$

Now if $(x_0,y_0)\notin V$, $\mathcal N(f)$ has two distinct eigenvalues which depend smoothly on the coefficients of $\mathcal N(f)$, hence on $(x,y)$ in a neighborhood $U$ of $(x,y)$. Moreover, for $(x,y)$ close to $(x_0,y_0)$, on can find two distinct non-zero eingenvectors $W_1(x,y)$ and $W_2(x,y)$ of $\mathcal N(f)(x,y)$ associated to its distinct eigenvalues, and depending smoothly on $(x,y)$. By definition, any solution $\gamma(t)=(x(t),y(t))\in U$ is such that $\gamma'(t)$ is colinear to $W_i(\gamma(t))$ for a certain $i$. Now two such solutions (for a fixed $i$) coincide up to reparametrization.
\end{proof}

It is now enough to understand under which conditions the solutions of the differential equation \eqref{equation:eq_diff_cones_renormalized} are supported by conics. The following lemma gives the result:

\begin{lemma}
\label{lemma:conical_solutions}
The following statements are equivalent:\\
(i) $b_1=c_1$;\\
(ii) Equation \eqref{equation:eq_diff_cones_renormalized} admits at least one solution which is supported by a conic. In this case, all solutions are supported on conics.
\end{lemma}

\begin{proof}
The proof will follow directly from a computation. We apologize in advance by readers who would like to see a more algebraic approach. We will indeed consider maps $x(t)$ and $y(t)$ given for all $t\in\RR$ by 
\begin{equation}
\label{equation:trigonometric_parametrization}
x(t) = r_1\cos(t)
\qquad\text{and}\qquad
y(t) = r_2\cos(t+\varphi)
\end{equation}
where $r_1,r_2\in\RR$ are non-zero and $\varphi\in\RR\setminus\pi\ZZ$ in the case of ellipses, and by 
\begin{equation}
\label{equation:hyperbolic_parametrization}
x(t) = r_1\cosh(t)
\qquad\text{and}\qquad
y(t) = r_2\cosh(t+\varphi)
\end{equation}
where $r_1,r_2,\varphi\in\RR$ are non-zero in the case of hyperbolae.

Suppose first that $(x,y)$ is given by \eqref{equation:trigonometric_parametrization}. Since Equation \eqref{equation:eq_diff_cones_renormalized} is invariant by the transformtion $(x,y)\mapsto(-x,-y)$, any possible solution lying on a conic is of the form $t\mapsto (x(\omega(t)),y(\omega(t)))$, where $\omega$ is a strictly monotone smooth function. Moreover, since Equation \eqref{equation:eq_diff_cones_renormalized} is homogenous in the derivatives, we can assume that $\omega(t) = t$. Thus it admits solutions supported on conics if and only if one can find such a pair $(x,y)$ solution of Equation \eqref{equation:eq_diff_cones_renormalized}.

Consider such a pair $(x,y)$ and write
$$E = \left(e+y^2-x^2\right)x'y'+(b+xy)x'^2-(c+xy)y'^2$$
where for simplicity we omit the index $1$ and we replaced $d_1-a_1$ by $e$. Substituting and factorizing, we obtain
\begin{multline}
E = er_1r_2\sin(t)\sin(t+\varphi)+br_1^2\sin(t)^2-cr_2^2\sin(t+\varphi)^2\\
+r_1^3r_2\sin(\varphi)\cos(t)\sin(t)
+r_1r_2^3\sin(\varphi)\cos(t+\varphi)\sin(t+\varphi).
\end{multline}
We will now write $E$ as a linear combination of the three independant maps $\sin^2$, $\cos^2$, $\sin\cos$. To do that, replacein $E$ the maps $\cos(t+\varphi)$ and $\sin(t+\varphi)$ by their expansions
$$\cos(t+\varphi) = \cos(t)\cos(\varphi)-\sin(t)\sin(\varphi),$$
$$\sin(t+\varphi) = \sin(t)\cos(\varphi)+\cos(t)\sin(\varphi).$$
We obtain
$$E =-K_1\cos(t)^2
+ K_2\sin(t)^2
+ K_3\cos(t)\sin(t)
$$
where 
$$
\begin{array}{ccl}
K_1 & = & r_1r_2^3\sin(\varphi)^2\cos(\varphi) + cr_2^2\sin(\varphi)^2\\
K_2 & = & br_1^2 + er_1r_2\cos(\varphi) - cr_2^2\cos(\varphi)^2 + r_1r_2^3\sin(\varphi)^2\cos(\varphi)\\
K_3 & = & er_1r_2\sin(\varphi) - r_1^3r_2\sin(\varphi) - 2cr_2^2\sin(\varphi)\cos(\varphi)\\
    &   & \qquad\qquad\qquad\qquad - r_1r_2^3\sin(\varphi)(\cos(\varphi)^2 - \sin(\varphi)^2)
\end{array}
$$
Hence the equation $E\equiv 0$ is equivalent to $K_1=K_2=K_3=0$. This system is not difficult to solve. Indeed, since $r_1,r_2\neq 0$ and $\varphi$ is not an integer multiple of $\pi$, we first notice that 
$$K_1=0\qquad\Leftrightarrow\qquad c=-r_1r_2\cos(\varphi).$$
Now in equation $K_3=0$, divide by $r_2\sin(\varphi)$ and replace the occurence of $c$ by $-r_1r_2\cos(\varphi)$ to obtain
$$
\left\{
\begin{array}{ccl}
K_1 & = & 0\\
K_3 & = & 0
\end{array}\right.
\Leftrightarrow
\left\{
\begin{array}{ccl}
c & = & -r_1r_2\cos(\varphi)\\
e & = & r_1^2-r_2^2
\end{array}\right..
$$
In $K_2$, replace the two occurences of $r_1r_2\cos(\varphi)$ by $-c$ to get $K_2 = br_1^2-cr_2^2-ec$, and use $e=r_1^2-r_2^2$ to obtain the following equivalence
\begin{equation}
\label{equation:final_system}
\left\{
\begin{array}{ccl}
K_1 & = & 0\\
K_2 & = & 0\\
K_3 & = & 0
\end{array}\right.
\Leftrightarrow
\left\{
\begin{array}{ccl}
c & = & -r_1r_2\cos(\varphi)\\
e & = & r_1^2-r_2^2\\
0 & = & (b-c)(r_2^2+e)
\end{array}\right..
\end{equation}
We deduce from the system \eqref{equation:final_system}, that if $b\neq c$, then $r_2^2+e=r_1^2=0$ which contradicts our assumptions. Otherwise, if 
$b=c$, all solutions of the form $t\mapsto (x(t),y(t))$ with 
$$\left\{
\begin{array}{ccl}
c & = & -r_1r_2\cos(\varphi)\\
e & = & r_1^2-r_2^2
\end{array}\right.$$
are solution of Equation \eqref{equation:eq_diff_cones_renormalized}. 

The case when $(x,y)$ is given by \eqref{equation:hyperbolic_parametrization} can be treated in the same way, with the same conclusion. Hence if Equation \eqref{equation:eq_diff_cones_renormalized} admits a solution which is supported by a conic (an ellipse or a hyperbola), then previous computations show that $b_1=c_1$. 

Conversely, assume $b_1=c_1$ and choose a point $(x_0,y_0)$ in $\RR^2$ for which one can find a solution $\gamma(t)=(x(t),y(t))$ of Equation \eqref{equation:eq_diff_cones_renormalized}. Consider the set $V$ of Lemma \ref{lemma:number_solutions}. Then in the case when $(x_0,y_0)$ doesn't lie in $V$, there are exactly two solutions to Equation \eqref{equation:eq_diff_cones_renormalized} up to reparametrization, and by previous computations one parametrizes an ellipse, the other one a hyperbola; hence so does $\gamma$. Now if $(x_0,y_0)$ lies in $V$, but $\gamma$ is not contained in $V$, we can use previous argument. Finally if $\gamma$ is contained in $V$, then since $V$ is the union of conics the result holds.
\end{proof}

\begin{proof}[Proof of Proposition \ref{proposition:cone_lines}]
It is a direct reformulation of Lemma \ref{lemma:conical_solutions}. First, note that $b_1=c_1$ is equivalent to $k_2b=k_1c$ (see Equation \eqref{equation:eq_diff_cones_renormalized}), hence 
$$k_2\scalar{d\nu(u_1)}{u_2}=k_1\scalar{d\nu(u_2)}{u_1}$$ 
where $(u_1,u_2)$ are the principal direction of $S$ at $q$. This gives exactly the $L$-symmetry assumption.

Now, a piece of cone is quadratic if and only its transverse section by a plane is an open set of a conic.

Finally, the fact that the set of lines tangent to $V$ and containing $q$ is included in the same cone as the one tangent to $U$ comes from the proof of Lemma \ref{lemma:conical_solutions}: let a plane $P$ parallel to $T_qS$, and intersecting the cone of lines tangent to $U$ transversally. We proved that this intersection is a conic $C$ which is centraly-symmetric with respect to the point of intersection of $P$ with the line $L_q$ defining the projective structure of $(S,L)$ at $q$. Hence the line tangent to $V$ and passing through $q$ $-$ which are reflected from the ones tangent to $U$ $-$ intersect $P$ on $C$, and the result follows.

Hence we showed that Proposition \ref{proposition:cone_lines} is equivalent to Lemma \ref{lemma:conical_solutions}, which gives the result.
\end{proof}

\section{Only quadrics have pseudo-caustics}
\label{section:theorem_only}

This Section is devoted to the proof of Theorem \ref{theorem:main_pseudo_Euclidean}. It comes from the projective structure that one can endow on pseudo-Euclidean billiards.

Consider $Q$ a (constant) non-degenerate quadratic form on $\RR^d$, $d\geq 3$, and $S\subset\RR^d$ an embedded hypersurface. For each $q\in S$, consider the line $L_q$ containing $q$ and orthogonal to $T_qS$ for the quadratic form $Q$. In the case when $L_q$ is transverse to $S$ for all $q$ (\textit{ie} $S$ is what we call a \textit{space-time hypersurface}), then $(S,L)$ is a line-framed hypersurface on which the projective law of reflection corresponds to the reflection of lines with respect to the quadratic form $Q$.

\begin{lemma}
\label{lemma:pseudo_euclidean_symmetric}
The hypersurface $S$ is $L$-symmetric.
\end{lemma}

\begin{proof}
Let $n:S\to\RR^3$ be a field of unit normal vector to $S$ for the Euclidean structure. Consider the $d$-dimensional invertible real matrix $M$ such that $M^T=M$ and for any $x,y\in\RR^d$ we have $Q(x,y)=\scalar{M^{-1}x}{y}$. For $q\in S$ define
$$\nu(q) = \frac{1}{\scalar{Mn(q)}{n(q)}}Mn(q).$$
The reader can check that $\nu(q)$ gives the direction of $L_q$ and satisfies $\scalar{\nu(q)}{n(q)}=1$. A simple computation shows that for any $x,y\in T_qS$
\begin{multline}
\scalar{d\nu_q(x)}{dn_q(y)} = 
\lambda(q)\scalar{Mdn_q(x)}{dn_q(y)}\\
-\lambda(q)^2\scalar{Mn(q)}{dn_q(x)}\scalar{Mn(q)}{dn_q(y)},
\end{multline}
from which we deduce 
$$\scalar{d\nu_q(x)}{dn_q(y)} = \scalar{d\nu_q(y)}{dn_q(x)}.$$
\end{proof}

We can now prove Theorem \ref{theorem:main_pseudo_Euclidean}. Let $S$ be an embedded space-time hypersurface together with its field of $Q$-orhtogonal lines $L$. Assume that $S$ admits a piece of caustics $(U,V)$ for the pseudo-Euclidean law of reflection, which amounts to say that $(U,V)$ is a piece of caustics for the projective law of reflection induced by $L$ on $S$.

By Lemma \ref{lemma:pseudo_euclidean_symmetric}, $S$ is $L$-symmetric. Hence by Theorem \ref{theorem:main_projective}, $U,V$ are open sets of one and the same quadric $\mathcal Q$, and any line tangent to $\mathcal Q$ intersecting $S$ transversally is reflected into a line tangent to $\mathcal{Q}$.

Let us construct now a distribution of hyperplanes on an open set of points containing $S$. For $q\in\RR^d$ sufficiently close to $S$, one can define the cone $\mathcal C_q$ of lines containing $q$ and tangent to $\mathcal Q$ such that the cone $\mathcal C_q$ is not contained in an affine hyperplane of $\RR^d$. There exists only one pair $(L_0(q),H(q))$ where $L_0(q)$ is a line containing $q$ which is $Q$-orthogonal to the hyperplane $H_0(q)$ and such that any line of $\mathcal C_q$ is reflected into a line in $\mathcal{C}_q$ by the reflection of lines on $H(q)$ with respect to the pseudo-Euclidean metric. We then consider the distribution $q\mapsto H(q)$ defined in a neighborhood of $S$.

It satisfies by construction $H(q) = T_qS$ and $L_0(q)=L_q$ for $q\in S$, hence $S$ is an integral hypersurface of $H$. Yet the integral hypersurfaces of this distribution are quadrics which are pseudo-confocal to $\mathcal Q$. Indeed, by \cite{KhesinTaba} quadrics in pseudo-confocal pencil of quadrics are caustics of one another. We deduce that $S$ is an open subset of a quadric which is pseudo-confocal to $\mathcal Q$. For a definition of a pencil of pseudo-confocal quadrics, see \cite{fierobe_these} Section 2.1.


\end{document}